\documentclass[12pt]{article}
\usepackage{mathtools}
\usepackage{amsfonts}
\usepackage{pst-node}
\usepackage{tikz-cd}

\usepackage{amsmath,amsfonts,amssymb,amsthm}
\input amssym.def
\newcommand{\Z}{{\mathbb Z}}
\newcommand{\C}{{\mathbb C}}

\newcommand{\N}{{\mathbb N}}

\def\<{\langle}
\def\>{\rangle}

\newtheorem{thm}{Theorem}[section]
\newtheorem{prop}[thm]{Proposition}
\newtheorem{lem}[thm]{Lemma}

\newtheorem{rmk}[thm]{Remark}
\newtheorem{definition}[thm]{Definition}

\newtheorem{example}[thm]{Example}

\input amssym.def
\begin{document}

\begin{center}
{\Large \bf  Algebra of $q$-difference operators, affine vertex algebras,
and their modules}
\end{center}

\begin{center}
	{Hongyan Guo
		\footnote{Partially supported by
			NSFC (No.11901224)
			 and NSF of Hubei Province (No.2019CFB160)
		}
		\\
		School of Mathematics and Statistics,
		Central China Normal University, \\ Wuhan 430079, China
	}
	
\end{center}

\begin{abstract}
	In this paper, we explore a canonical connection between the algebra of $q$-difference operators $\widetilde{V}_{q}$,
affine Lie algebra and
affine vertex algebras associated to certain subalgebra $\mathcal{A}$
of the Lie algebra
$\mathfrak{gl}_{\infty}$.
We also introduce and study a category $\mathcal{O}$
 of $\widetilde{V}_{q}$-modules.
More precisely,
  we obtain a realization of $\widetilde{V}_{q}$
 as a covariant algebra of the affine Lie algebra $\widehat{\mathcal{A}^{*}}$,
 where $\mathcal{A}^{*}$ is a 1-dimensional central extension of
 $\mathcal{A}$.
 We prove that restricted $\widetilde{V_{q}}$-
 modules of level $\ell_{12}$
 correspond to $\Z$-equivariant $\phi$-coordinated quasi-modules
 for the vertex algebra $V_{\widetilde{\mathcal{A}}}(\ell_{12},0)$,
 where $\widetilde{\mathcal{A}}$ is a generalized affine Lie algebra of $\mathcal{A}$.
 In the end, we show that objects in the
 category $\mathcal{O}$ are restricted $\widetilde{V_{q}}$-modules,
  and we classify simple modules in the category $\mathcal{O}$.
	
	\end{abstract}

\section{Introduction}
\def\theequation{1.\arabic{equation}}
\setcounter{equation}{0}

The algebra of $q$-difference operators, also named $q$-analog Virasoro-like algebra, has a lot of background. It can be viewed as a $q$-deformation of Virasoro-like algebra (cf. \cite{KPS94}). It is the (2-dimensional) universal central
extension of the Lie algebra of inner derivations of the quantum torus $\C_{q}[x^{\pm 1}, y^{\pm 1}]$ (cf. \cite{BGK96}, \cite{KPS94}, etc.).
It can also
be viewed as a Lie algebra whose universal enveloping algebra is the toroidal algebra which correspond to the quantum toroidal $\mathfrak{gl}_{1}$ algebra $U_{q,t}(\ddot{\mathfrak{gl}}_{1})$ at $q=t$ (cf. \cite{BGK96}, \cite{BG20}, etc.).
Certain deformation of the universal enveloping algebra of the algebra of $q$-difference operators has been introduced and studied in \cite{M07}, and it has a connection with the deformed Virasoro algebra and $W_{N}$ algebras. The structure theory and representation theory of the algebra of $q$-difference operators has been widely studied
(\cite{BG20}, \cite{LT06}, \cite{LT08},  \cite{MJ98}, \cite{ZZ96}, to name a few.).

More explicitly, the algebra of $q$-difference operators $\widetilde{V}_{q}$ is a Lie algebra spanned by the elements $E_{k,l}$, $c_{1}$, $c_{2}$, where $(k,l)\in\mathbb{Z}^{2}$, $c_{1}$ and $c_{2}$
 		are central elements,
 		with Lie bracket
 		\begin{eqnarray*}
 			[E_{k,l},E_{r,s}]=(q^{rl-sk}-q^{sk-rl})E_{k+r,l+s}
 			+\delta_{k,-r}\delta_{l,-s}(kc_{1}+lc_{2}),
 		\end{eqnarray*}
 		for $(k,l), (r,s)\in\mathbb{Z}^{2}.$
 We add $E_{0,0}=0$ in the definition.
 This paper is twofold.
On the one hand, we study the algebra of $q$-difference operators from the
vertex algebras points of view.
We
build a connection between restricted modules of the algebra of $q$-difference operators
and certain modules of generalized affine vertex algebras.
On the other hand, we study a category of $\widetilde{V}_{q}$-modules which are locally finite with respect to some
 subalgebra of $\widetilde{V}_{q}$. Modules in the category are proved to be restricted $\widetilde{V}_{q}$-modules and simple modules are classified.

 Constructing vertex algebras from Lie algebras and build a connection
 between the representations of the two algebraic objects is an important
 subject in both Lie algebra and vertex algebra areas.
 For certain Lie algebras $\mathcal{L}$, like Virasoro algebra, Heisenberg algebra
 and affine Lie algebras, etc.,
 the corresponding vertex algebras (denote by $V_{\mathcal{L}}$)
 are the vacuum modules of $\mathcal{L}$,
 and restricted $\mathcal{L}$-modules at fixed levels
  are in one-to-one correspondence to
 $V_{\mathcal{L}}$-modules (cf. \cite{FZ92}, \cite{LL04}, etc.).
 However, there are Lie algebras
 (by abuse of notation, we still denote by $\mathcal{L}$)
  whose restricted modules
 are not related to the vertex algebras $V_{\mathcal{L}}$,
 but vertex algebras constructed from other Lie algebras (cf. \cite{GLTW1},
 \cite{GLTW2}, \cite{LTW20}, etc.).
 And restricted $\mathcal{L}$-modules are not correspond to vertex algebra modules, but rather quasi modules or
 (equivariant) $\phi$-coordinated (quasi) modules which are introduced
 and studied
 in a series of papers (\cite{L06}, \cite{L07}, \cite{L11}, \cite{L13}, etc.).

Affine Kac-Moody algebra $\widehat{\mathfrak{g}}$
is the 1-dimensional universal central
extension of the loop algebra (also called current algebra)
$\mathfrak{g}\otimes\C[t,t^{-1}]$, where $\mathfrak{g}$
is a finite-dimensional simple Lie algebra.
Affine vertex algebras $V_{\widehat{\mathfrak{g}}}(\ell,0)$
are constructed from affine Lie algebras
(with $\mathfrak{g}$ possibly infinite-dimensional).
The relations between the
representations of affine Lie algebra $\widehat{\mathfrak{g}}$
and the representations of affine vertex algebras $V_{\widehat{\mathfrak{g}}}(\ell,0)$
are investigated by many authors
(cf. \cite{FZ92}, \cite{LL04}, \cite{L07}, etc.).
In trying to relate
the algebra of $q$-difference operators with vertex algebras and their modules,
we need to consider the following generalized affine Lie algebras.

Let $\mathfrak{g}$ be a (possibly infinite-dimensional) Lie algebra equipped
with a symmetric invariant bilinear form $\langle\;,\;\rangle$.
Suppose that $\psi:\mathfrak{g}\times\mathfrak{g}\longrightarrow\C$ is a 2-cocycle on $\mathfrak{g}$.
Consider the following generalized affine Lie algebra
$$\widetilde{\mathfrak{g}}=\mathfrak{g}\otimes\C[t,t^{-1}]\oplus\C K_{1}\oplus\C K_{2}$$ with Lie bracket
\begin{eqnarray*}
[a\otimes t^{i},b\otimes t^{j}]=[a,b]\otimes t^{i+j}+\psi(a,b)\delta_{i+j+1,0}K_{1}+i\langle a,b\rangle\delta_{i+j,0}K_{2},
\end{eqnarray*}
where $K_{1}$, $K_{2}$ are central elements, $a,b\in\mathfrak{g}$, $i,j\in\Z$.
 It is easy to see that $\widetilde{\mathfrak{g}}$ can also be viewed as a 1-dimensional central extension of the (usual) affine Lie algebra $\widehat{\mathfrak{g}}=\mathfrak{g}\otimes\C[t,t^{-1}]\oplus\C K_{2}$
 (see Remark \ref{p2}). The corresponding vertex algebras $V_{\widetilde{\mathfrak{g}}}(\ell_{12},0)$
 are straightforward to construct.
 Restricted modules of the Lie algebra $\widetilde{\mathfrak{g}}$
 and modules of the vertex algebra $V_{\widetilde{\mathfrak{g}}}(\ell_{12},0)$
 are closely related.
 The relations of the generalized affine vertex algebra $V_{\widetilde{\mathfrak{g}}}(\ell_{12},0)$
 and affine vertex algebra $V_{\widehat{\mathfrak{g}}}(\ell_{2},0)$
 are discussed in this paper.

 Let $\Gamma$ be a group of automorphisms of $\mathfrak{g}$, preserving the bilinear form and satisfying $[ga,b]=0=\langle ga,b\rangle$ for all but finitely many
 $g\in\Gamma$. Let $\chi:\Gamma\longrightarrow\C^{\times}$ be
 a group homomorphism.
 The covariant algebra $\widehat{\mathfrak{g}}[\Gamma]$ which
 is a quotient space of the affine Lie algebra $\widehat{\mathfrak{g}}$
 has been construct and studied in \cite{L07} (see also \cite{G-KK98}).
 $\widehat{\mathfrak{g}}[\Gamma]$ is a generalization of twisted affine Lie algebras.
 If $\chi$ is injective, then restricted modules of level $\ell$ for $\widehat{\mathfrak{g}}[\Gamma]$ correspond
 to quasi-modules for $V_{\widehat{\mathfrak{g}}}(\ell,0)$ viewed as a
  $\Gamma$-vertex algebra (Theorem 4.9 of \cite{L07}).

 In the study of $q$-Virasoro algebra and trigonometric Lie algebras,
certain Lie algebra $\mathcal{A}$ plays an important role
 (cf. \cite{GLTW1}, \cite{GLTW2}, \cite{LTW20}).
 More specifically, $\mathcal{A}$ is a subalgebra of the Lie algebra $\mathfrak{gl}_{\infty}$  which is spanned
 by the basis elements (the elementary matrices) $E_{m,n}$ with $m,n\in\Z$, $m+n\in2\Z$.
 Rewrite the basis elements as $G_{\alpha,m}=E_{m+\alpha,m-\alpha}$
 for $\alpha,m\in\Z$. Then $\mathcal{A}$
 is spanned by $G_{\alpha,m}$, $\alpha,m\in\Z$.
 Let $\alpha,\beta,m,n\in\Z$.
 There is a natural symmetric invariant bilinear form
 on $\mathcal{A}$:
 $$\langle G_{\alpha,m}, G_{\beta,n}\rangle=\delta_{\alpha+\beta,0}\delta_{m-n,0}.$$
  Consider the 2-cocycle $\psi$ on $\mathcal{A}$ defined
  by
  	$$\psi(G_{\alpha,m}, G_{\beta,n})=\alpha\delta_{\alpha+\beta,0}\delta_{m-n,0}.$$
  Even though $\psi$ is proved to be a trivial 2-cocycle (Proposition \ref{dirsum}),
  in order to relate the algebra of $q$-difference operators $\widetilde{V_{q}}$ with vertex algebras we need to consider this 2-cocycle.
  We first show that $\widetilde{V_{q}}$ is
  isomorphic to a covariant algebra of the affine Lie algebra
  $\widehat{\mathcal{A}^{*}}$, where $\mathcal{A}^{*}$ is the 1-dimensional
  central extension of $\mathcal{A}$ via the 2-cocycle $\psi$.
  Then, for any $\ell_{1},\ell_{2}\in\C$, we establish an equivalence between the category of restricted $\widetilde{V_{q}}$-modules of level $\ell_{12}$ (means that $c_{1}, c_{2}$ act as scalars $\ell_{1},\ell_{2}$ respectively)
  and the category of $\Z$-equivariant
  $\phi$-coordinated quasi modules for the vertex algebra $V_{\widetilde{\mathcal{A}}}(\ell_{12},0)$.

 Let $t$ be any positive integer. Consider the subalgebra
 $$\widetilde{V_{q}}^{(t)}=\sum\limits_{k\in\Z,l\geq t}\C E_{k,l}.$$
 To study the restricted modules for the algebra of $q$-difference
  operators, we introduce a category $\mathcal{O}$ of $\widetilde{V_{q}}$-modules.
  It is a subcategory of modules with the property that they
   are locally finite $\widetilde{V_{q}}^{(t)}$-modules for some $t$.
  We prove that objects in the category $\mathcal{O}$
are restricted $\widetilde{V}_{q}$-modules.

Finally, we classify simple modules in the category $\mathcal{O}$.
Let
$$b=
\bigoplus_{l\geq 0,k\in\mathbb{Z}}\mathbb{C}E_{k,l}
\oplus\mathbb{C} c_{1}\oplus\mathbb{C} c_{2}.$$
Consider the induced module
$$\mbox{Ind}_{\ell_{12}}(V)=U(\widetilde{V_{q}})\otimes_{U(b)}V,$$
where $V$ is a $b$-module with $c_{1},c_{2}$ act as $\ell_{1},\ell_{2}$.
We first show that if $V$ is a simple $b$-module, then under
certain conditions $\mbox{Ind}_{\ell_{12}}(V)$ is a simple
$\widetilde{V_{q}}$-module.
Then we prove that simple modules in the category $\mathcal{O}$
are either
highest weight modules or induced modules $\mbox{Ind}_{\ell_{12}}V$
for some simple $b$-module $V$.

  This paper is organized as follows.
  In Section \ref{sec:2}, we study generalized affine Lie algebras and
  the corresponding affine vertex algebras as well as their modules.
  In Section \ref{sec:3}, we give a realization of
  the algebra of $q$-difference operators $\widetilde{V_{q}}$
  in terms of the covariant algebra of the affine Lie algebra $\widehat{\mathcal{A}^{*}}$.
  In Section \ref{sec:4}, we show that restricted $\widetilde{V_{q}}$-module category of level $\ell_{12}$
   is equivalent to the $\Z$-equivariant
  $\phi$-coordinated quasi module category for the generalized affine vertex algebra
   $V_{\widetilde{\mathcal{A}}}(\ell_{12},0)$.
   In Section \ref{sec:5}, we introduce and study a category $\mathcal{O}$
of $\widetilde{V}_{q}$-modules.
We prove that all the modules in the category $\mathcal{O}$
are restricted $\widetilde{V}_{q}$-modules.
We also give a classification of simple modules in the category $\mathcal{O}$.

Throughout this paper, we denote by $\Z$, $\N$, $\mathbb{Z}_{+}$ and $\C$
the sets of integers, nonnegative integers, positive integers and complex numbers
respectively.

\section{Generalized affine vertex algebras and their representations}
\def\theequation{2.\arabic{equation}}
\label{sec:2}
\setcounter{equation}{0}

In this section, we construct and study
certain generalized affine Lie algebras $\widetilde{\mathfrak{g}}$ and
the corresponding generalized affine vertex algebras
$V_{\widetilde{\mathfrak{g}}}(\ell_{12},0)$.
We also discuss the relations between $V_{\widetilde{\mathfrak{g}}}(\ell_{12},0)$
and affine vertex algebra $V_{\widehat{\mathfrak{g}}}(\ell_{2},0)$.

\begin{definition}
{\em
Let $\mathfrak{g}$ be a Lie algebra.
A {\em 2-cocycle} $\psi:\mathfrak{g}\times\mathfrak{g}\longrightarrow\C$ on $\mathfrak{g}$ is a skew-symmetric bilinear form satisfying
$\psi(x,[y,z])+\psi(y,[z,x])+\psi(z,[x,y])=0$ for any $x,y,z\in\mathfrak{g}$.

}
\end{definition}

Let $\mathfrak{g}$ be a (possibly infinite-dimensional) Lie algebra over the complex number field $\C$.
Let $\langle\;,\;\rangle:\mathfrak{g}\times\mathfrak{g}\longrightarrow\C$ be a symmetric invariant bilinear form on $\mathfrak{g}$.
Suppose that
$\psi:\mathfrak{g}\times\mathfrak{g}\longrightarrow\C$ is a 2-cocycle on $\mathfrak{g}$.

With respect to \Big($\mathfrak{g}$, $\langle\;,\;\rangle$, $\psi$\Big),
there are the following three ways that we can get affine-type Lie algebras:
\begin{itemize}
\item[(1)]The affine Lie algebra $\widehat{\mathfrak{g}}=\mathfrak{g}\otimes\C[t,t^{-1}]\oplus\C K_{2}$ with Lie bracket
    \begin{eqnarray}
[a\otimes t^{i},b\otimes t^{j}]=[a,b]\otimes t^{i+j}+i\langle a,b\rangle\delta_{i+j,0}K_{2},
\end{eqnarray}
where $K_{2}$ is a central element, $a,b\in\mathfrak{g}$, $i,j\in\Z$.

\item[(2)] Let $\mathfrak{g}^{*}=\mathfrak{g}\oplus\C K_{1}$ be
 the 1-dimensional central extension of $\mathfrak{g}$ by $K_{1}$ via $\psi$. For $a,b\in\mathfrak{g}$, $\lambda,\mu\in\C$,
    the Lie bracket on $\mathfrak{g}^{*}$ is given by
    \begin{eqnarray}
    [a+\lambda K_{1},b+\mu K_{1}]=[a,b]+\psi(a,b)K_{1}.
    \end{eqnarray}
    The bilinear form $\langle\;,\;\rangle$ on $\mathfrak{g}$ can be naturally extended to a symmetric invariant bilinear form on $\mathfrak{g}^{*}$ with
    $\langle\mathfrak{g}^{*},K_{1}\rangle=\langle K_{1}, \mathfrak{g}^{*}\rangle=0$.
    Then we have the affine Lie algebra
    $\widehat{\mathfrak{g}^{*}}=\mathfrak{g}^{*}\otimes\C[t,t^{-1}]\oplus\C K_{2}$ with Lie bracket
\begin{eqnarray}
[a^{*}\otimes t^{i},b^{*}\otimes t^{j}]=[a^{*},b^{*}]\otimes t^{i+j}
+i\langle a^{*},b^{*}\rangle\delta_{i+j,0}K_{2},
\end{eqnarray}
where $K_{2}$ is a central element, $a^{*},b^{*}\in\mathfrak{g}^{*}$, $i,j\in\Z$.
Clearly, $K_{1}\otimes\C[t,t^{-1}]$ is also in the center of the affine Lie algebra $\widehat{\mathfrak{g}^{*}}$.

\item[(3)] The generalized affine Lie algebra $\widetilde{\mathfrak{g}}=\mathfrak{g}\otimes\C[t,t^{-1}]\oplus\C K_{1}\oplus\C K_{2}$ with Lie bracket
\begin{eqnarray}
[a\otimes t^{i},b\otimes t^{j}]=[a,b]\otimes t^{i+j}+\psi(a,b)\delta_{i+j+1,0}K_{1}+i\langle a,b\rangle\delta_{i+j,0}K_{2},
\label{tilde}
\end{eqnarray}
where $K_{1}$, $K_{2}$ are central elements, $a,b\in\mathfrak{g}$, $i,j\in\Z$.
\end{itemize}

\begin{rmk}
\label{p2}
{\em
\begin{itemize}
\item[(1)] The 2-cocycle $\psi$ can be extended to a 2-cocycle $\psi_{2}$ on
$\widehat{\mathfrak{g}}$ with
$$\psi_{2}(a\otimes t^{i},b\otimes t^{j})=\psi(a,b)\delta_{i+j+1,0},$$
$$\psi_{2}(\widehat{\mathfrak{g}},K_{2})=\psi_{2}(K_{2},\widehat{\mathfrak{g}})=0$$
for $a,b\in\mathfrak{g}$, $i,j\in\Z$.
Then $\widetilde{\mathfrak{g}}$ can be viewed as a 1-dimensional central extension of the affine Lie algebra
$\widehat{\mathfrak{g}}$ by $K_{1}$ via $\psi_{2}$. More precisely,
we can write $\widetilde{\mathfrak{g}}=\widehat{\mathfrak{g}}\oplus\C K_{1}$ with Lie bracket
$$[\widehat{a}+\lambda K_{1},\widehat{b}+\mu K_{1}]
=[\widehat{a},\widehat{b}]+\psi_{2}(\widehat{a},\widehat{b})K_{1},$$
where $\widehat{a},\widehat{b}\in\widehat{\mathfrak{g}}$, $\lambda,\mu\in \C$.
Note that $\widehat{\mathfrak{g}}$ may not be a subalgebra of
$\widetilde{\mathfrak{g}}$.
\item[(2)] It is easy to see that
$\widehat{\mathfrak{g}^{*}}/(K_{1}\otimes\C[t,t^{-1}])\cong \widehat{\mathfrak{g}}$
and $\widetilde{\mathfrak{g}}/\C K_{1}\cong \widehat{\mathfrak{g}}$ as Lie algebras.

\end{itemize}
}
\end{rmk}

In the following, we consider the (general) affine Lie algebra $\widetilde{\mathfrak{g}}$.
For $a\in\mathfrak{g}$, form a generating function
$$a(x)=\sum_{n\in\Z}a(n)x^{-n-1},$$
where $a(n)$ stands for $a\otimes t^{n}$.
The defining relation (\ref{tilde}) of $\widetilde{\mathfrak{g}}$ can be written as
\begin{eqnarray}
&&{}
[a(x_{1}),b(x_{2})]
=[a,b](x_{2})x_{1}^{-1}\delta\Big(\frac{x_{2}}{x_{1}}\Big)
\nonumber\\
&&{}\;\;\;\;+\psi(a,b)x_{1}^{-1}\delta\Big(\frac{x_{2}}{x_{1}}\Big)K_{1}
+\langle a,b\rangle\frac{\partial}{\partial x_{2}}x_{1}^{-1}\delta\Big(\frac{x_{2}}{x_{1}}\Big)K_{2} \label{eq:1}
\end{eqnarray}
for $a,b\in\mathfrak{g}$.

Let $\ell_{1},\ell_{2}$ be complex numbers.
View $\C$ as a module for the subalgebra $\mathfrak{g}\otimes\C[t]+\C K_{1}+\C K_{2}$ of $\widetilde{\mathfrak{g}}$
with $\mathfrak{g}\otimes\C[t]$ acting trivially and with $K_{1}$, $K_{2}$ acting as scalars $\ell_{1},\ell_{2}$ respectively.
Then form an induced module
$$V_{\widetilde{\mathfrak{g}}}(\ell_{12},0)
=U(\widetilde{\mathfrak{g}})\otimes_{U(\mathfrak{g}\otimes\C[t]+\C K_{1}+\C K_{2})}\C,$$
where $U(\cdot)$ denotes the universal enveloping algebra of a Lie algebra.

Set ${\bf 1}=1\otimes 1\in V_{\widetilde{\mathfrak{g}}}(\ell_{12},0)$
and identify $\mathfrak{g}$ as a subspace of $V_{\widetilde{\mathfrak{g}}}(\ell_{12},0)$
through linear map
$$\mathfrak{g}\longrightarrow V_{\widetilde{\mathfrak{g}}}(\ell_{12},0);
\;\;a\mapsto a(-1){\bf 1}.$$
Then there is a unique vertex algebra structure on
$(V_{\widetilde{\mathfrak{g}}}(\ell_{12},0),Y,{\bf 1})$ given by $Y(a,x)=a(x)$ for $a\in\mathfrak{g}$.
Furthermore, $V_{\widetilde{\mathfrak{g}}}(\ell_{12},0)$ is a $\Z$-graded
vertex algebra:
\begin{eqnarray}
V_{\widetilde{\mathfrak{g}}}(\ell_{12},0)=\coprod_{n\geq 0}V_{\widetilde{\mathfrak{g}}}(\ell_{12},0)_{(n)},\label{eq:3}
\end{eqnarray}
where $V_{\widetilde{\mathfrak{g}}}(\ell_{12},0)_{(n)}$ is spanned by the vectors
$a^{1}(-m_{1})\cdots a^{r}(-m_{r}){\bf 1}$ for
$r\geq 0$, $a^{i}\in\mathfrak{g}$, $m_{i}\geq 1$ with $m_{1}+\cdots+m_{r}=n$.

For a vector space $W$, set
$$\mathcal{E}(W)=\mbox{Hom}(W,W((x)))\subset(\mbox{End} W)[[x,x^{-1}]].$$
\begin{definition}
{\em
We say that a $\widetilde{\mathfrak{g}}$-module $W$ is {\em restricted} if for any $w\in W$, $a\in\mathfrak{g}$,
$a(n)w=0$ for $n$ sufficiently large, or equivalently, $a(x)\in\mathcal{E}(W)$. We say $W$ is of {\em level} $\ell_{12}$ if the central element $K_{i}$ acts as
scalar $\ell_{i}$, $i=1,2$.
}
\end{definition}

As usual, there is the following result
(cf. Theorem 6.2.12 and Theorem 6.2.13 of \cite{LL04}).
\begin{thm}
\label{cor1}
Let $\ell_{1},\ell_{2}$ be any complex numbers.
Let $W$ be any restricted $\widetilde{\mathfrak{g}}$-module
of level $\ell_{12}$.
Then there exists a unique $V_{\widetilde{\mathfrak{g}}}(\ell_{12},0)$-module structure
on $W$ such that for $a\in\mathfrak{g}$, $Y_{W}(a,x)=a_{W}(x)(=\sum_{n\in\Z}a(n)x^{-n-1})$.
Conversely, any module $W$ for the vertex algebra $V_{\widetilde{\mathfrak{g}}}(\ell_{12},0)$ is naturally
a restricted $\widetilde{\mathfrak{g}}$-module
of level $\ell_{12}$, with $a_{W}(x)=Y_{W}(a,x)$ for $a\in\mathfrak{g}$.
Furthermore, the
$V_{\widetilde{\mathfrak{g}}}(\ell_{12},0)$-submodules of $W$ coincide with the $\widetilde{\mathfrak{g}}$-submodules
of $W$.
\end{thm}

For affine Lie algebras $\widehat{\mathfrak{g}}$ and $\widehat{\mathfrak{g}^{*}}$,
we denote by $V_{\widehat{\mathfrak{g}}}(\ell_{2},0)$ and
$V_{\widehat{\mathfrak{g}^{*}}}(\ell_{2},0)$ their corresponding affine vertex algebras
with $K_{2}$ acts as $\ell_{2}$ (cf. \cite{FZ92}, \cite{LL04}, etc.).

Let $J_{\widehat{\mathfrak{g}}}(\ell_{2},0)$,
$J_{\widehat{\mathfrak{g}^{*}}}(\ell_{2},0)$,
$J_{\widetilde{\mathfrak{g}}}(\ell_{12},0)$ be the maximal (two-sided) ideals
of the vertex algebras $V_{\widehat{\mathfrak{g}}}(\ell_{2},0)$,
$V_{\widehat{\mathfrak{g}^{*}}}(\ell_{2},0)$,
$V_{\widetilde{\mathfrak{g}}}(\ell_{12},0)$ respectively.
Denote by
 $L_{\widetilde{\mathfrak{g}}}(\ell_{12},0)=
 V_{\widetilde{\mathfrak{g}}}(\ell_{12},0)/J_{\widetilde{\mathfrak{g}}}(\ell_{12},0)$,
 $L_{\widehat{\mathfrak{g}^{*}}}(\ell_{2},0)
=V_{\widehat{\mathfrak{g}^{*}}}(\ell_{2},0)/J_{\widehat{\mathfrak{g}^{*}}}(\ell_{2},0)$,
and  $L_{\widehat{\mathfrak{g}}}(\ell_{2},0)
=V_{\widehat{\mathfrak{g}}}(\ell_{2},0)/J_{\widehat{\mathfrak{g}}}(\ell_{2},0)$
 the corresponding simple vertex algebras.

We now investigate the relations between
the three types of affine vertex algebras $V_{\widehat{\mathfrak{g}}}(\ell_{2},0)$,
$V_{\widehat{\mathfrak{g}^{*}}}(\ell_{2},0)$, and
$V_{\widetilde{\mathfrak{g}}}(\ell_{12},0)$
as well as their simple quotients.

\begin{rmk}
{\em	There is a natural surjective vertex algebra homomorphism
	$
		\varphi:V_{\widehat{\mathfrak{g}^{*}}}(\ell_{2},0)\longrightarrow V_{\widehat{\mathfrak{g}}}(\ell_{2},0)$ with $K_{1}(-m){\bf 1}\mapsto 0$ for $m\geq 1$.
Then
$\mbox{Ker}(\varphi)=\langle K_{1}(-m){\bf 1}, m\geq 1\rangle$. It induces a vertex algebra isomorphism
	\begin{eqnarray}
		V_{\widehat{\mathfrak{g}^{*}}}(\ell_{2},0)\slash \langle K_{1}(-m){\bf 1}, m\geq 1\rangle\cong V_{\widehat{\mathfrak{g}}}(\ell_{2},0). \label{vaiso}
	\end{eqnarray}
And
\begin{eqnarray}
	L_{\widehat{\mathfrak{g}^{*}}}(\ell_{2},0)\cong L_{\widehat{\mathfrak{g}}}(\ell_{2},0)
\end{eqnarray}
as simple vertex algebras, for any $\ell_{2}\in\C$.
}
\end{rmk}

Clearly, if $\ell_{1}=0$, then
$$V_{\widetilde{\mathfrak{g}}}(\ell_{12},0)= V_{\widehat{\mathfrak{g}}}(\ell_{2},0)
\;\;\mbox{and}\;\;
L_{\widetilde{\mathfrak{g}}}(\ell_{12},0)= L_{\widehat{\mathfrak{g}}}(\ell_{2},0)$$
as vertex algebras.

\begin{definition}
{\em A 2-cocycle $\psi:\mathfrak{g}\times\mathfrak{g}\longrightarrow\C$
is said to be {\em trivial} if there exists a linear map
$\mu:\mathfrak{g}\longrightarrow\C$ such that for any $x,y\in\mathfrak{g}$,
$$\psi(x,y)=\mu([x,y]).$$
}
\end{definition}

\begin{prop}
\label{tcoyc}
Suppose that $\psi:\mathfrak{g}\times\mathfrak{g}\longrightarrow\C$ is a trivial 2-cocycle. Then $\mathfrak{g}^{*}\cong\mathfrak{g}\oplus\C K_{1}$
and $\widetilde{\mathfrak{g}}\cong\widehat{\mathfrak{g}}\oplus\C K_{1}$
as Lie algebras.
\end{prop}
\begin{proof}
Let $\mu:\mathfrak{g}\longrightarrow\C$ be a linear map such
that $\psi(x,y)=\mu([x,y])$ for any $x,y\in\mathfrak{g}$.
It is easy to check that
$f:\mathfrak{g}\oplus\C K_{1}\longrightarrow\mathfrak{g}^{*}$,
$(a,\lambda K_{1})\mapsto a+(\mu(a)+\lambda)K_{1}$,
$a\in\mathfrak{g}$, $\lambda\in\C$,
is a Lie algebra isomorphism.

Note that the 2-cocycle $\psi_{2}$ in Remark \ref{p2} then is also a trivial 2-cocycle.
Let $\mu_{2}:\widehat{\mathfrak{g}}\longrightarrow\C$ be a linear map defined by
$\mu_{2}(a\otimes t^{i}+\lambda K_{2})=\mu(a)\delta_{i+1,0}$ for any $a\in\mathfrak{g}$,
$i\in\Z$, $\lambda\in\C$.
Then $\psi_{2}(\widehat{a},\widehat{b})=\mu_{2}([\widehat{a},\widehat{b}])$ for any $\widehat{a},\widehat{b}\in\widehat{\mathfrak{g}}$.
Similarly, $f_{2}:\widehat{\mathfrak{g}}\oplus
\C K_{1}\longrightarrow\widetilde{\mathfrak{g}}$,
$(\widehat{a},\lambda K_{1})\mapsto \widehat{a}+(\mu_{2}(\widehat{a})+\lambda)K_{1}$,
$\widehat{a}\in\widehat{\mathfrak{g}}$, $\lambda\in\C$,
is a Lie algebra isomorphism.
\end{proof}

Then we can get the following assertion.
\begin{prop}
\label{iso}
Let $\ell_{1}$ be any nonzero complex number.
If $\psi$ is a trivial 2-cocycle on $\mathfrak{g}$.
Then we have
$$V_{\widetilde{\mathfrak{g}}}(\ell_{12},0)\cong V_{\widehat{\mathfrak{g}}}(\ell_{2},0)$$
as vertex algebras, for any $\ell_{2}\in\C$.
Hence also $L_{\widetilde{\mathfrak{g}}}(\ell_{12},0)\cong L_{\widehat{\mathfrak{g}}}(\ell_{2},0)$ for any $\ell_{2}\in\C$.
\end{prop}
\begin{proof}
$V_{\widehat{\mathfrak{g}}}(\ell_{2},0)$ is naturally a $\widehat{\mathfrak{g}}$-module.
It can be viewed as a $\widehat{\mathfrak{g}}\oplus\C K_{1}$-module with
$K_{1}$ acts as a scalar $\ell_{1}$. Explicitly,
$(\widehat{a},\lambda K_{1}).w=\widehat{a}.w+\lambda\ell_{1}w$ for
$\widehat{a}\in\widehat{\mathfrak{g}}$, $\lambda\in\C$,
$w\in V_{\widehat{\mathfrak{g}}}(\ell_{2},0)$.
By Proposition \ref{tcoyc},
$V_{\widehat{\mathfrak{g}}}(\ell_{2},0)$ is then a $\widetilde{\mathfrak{g}}$-module
with $a\otimes t^{i}$ acts as $(a\otimes t^{i},-\mu(a)\delta_{i+1,0}K_{1})$,
$K_{1}$ acts as $\ell_{1}$, $K_{2}$ acts as $\ell_{2}$, for any $a\in\mathfrak{g}$, $i\in\Z$.

Let $\varphi:V_{\widetilde{\mathfrak{g}}}(\ell_{12},0)\longrightarrow V_{\widehat{\mathfrak{g}}}(\ell_{2},0)$ be a $\widetilde{\mathfrak{g}}$-module
homomorphism with $\varphi({\bf 1})={\bf 1}$.
Then
$$\varphi(a(-1){\bf 1})=a(-1).{\bf 1}=a(-1){\bf 1}-\mu(a)\ell_{1}{\bf 1}$$
for any $a\in\mathfrak{g}$.
It is straightforward to check that $\varphi(a_{n}b)=\varphi(a)_{n}\varphi(b)$
for any $a,b\in\mathfrak{g}$, $n\in\Z$ (note that ${\bf 1}_{n}b=\delta_{n+1,0}b$).
Hence $\varphi$ is a vertex algebra homomorphism.

Similarly, $V_{\widetilde{\mathfrak{g}}}(\ell_{12},0)$ can be viewed as a
$\widehat{\mathfrak{g}}$-module with $K_{2}$ acts as $\ell_{2}$, $a\otimes t^{i}$
acts as $a\otimes t^{i}+\mu(a)\delta_{i+1,0}\ell_{1}$ for any $a\in\mathfrak{g}$,
$i\in\Z$.
Let $\Phi:V_{\widehat{\mathfrak{g}}}(\ell_{2},0)\longrightarrow V_{\widetilde{\mathfrak{g}}}(\ell_{12},0)$ be a $\widehat{\mathfrak{g}}$-module
homomorphism with $\Phi({\bf 1})={\bf 1}$. Then $\varphi$ is a vertex algebra
isomorphism with inverse $\Phi$.

\end{proof}

At the end of this section, we give two examples.
The second example plays an important role in our study of
the algebra of $q$-difference operators.

\begin{example}
\label{exmp1}
{\em $\mathfrak{g}=\mathfrak{gl}_{\infty}$ is the Lie algebra of doubly infinite complex
matrices with only finitely many nonzero entries under the usual commutator bracket.
For any $m,n\in\Z$, denote by $E_{m,n}$ the matrix whose only nonzero entry is the $(m,n)$-entry which is 1.
 Then $E_{m,n}$, $m,n\in\Z$, form a bases of $\mathfrak{gl}_{\infty}$.
Let $m,n,p,q\in\Z$.
The Lie bracket on $\mathfrak{gl}_{\infty}$ is given by
$$[E_{m,n}, E_{p,q}]=\delta_{n,p}E_{m,q}-\delta_{q,m}E_{p,n}.$$
$\mathfrak{gl}_{\infty}$ is equipped with a natural symmetric invariant bilinear form:
$$\langle E_{m,n}, E_{p,q}\rangle=\mbox{trace}(E_{m,n}E_{p,q})=\delta_{m,q}\delta_{n,p}.$$
There is a 2-cocycle $\psi$ on the Lie algebra $\mathfrak{gl}_{\infty}$ defined by
$$\psi(E_{m,n}, E_{n,m})=1=-\psi(E_{n,m}, E_{m,n})
\;\;\mbox{if}\;m\leq 0 \;\mbox{and}\;n\geq 1,$$
$$\psi(E_{m,n}, E_{p,q})=0\;\;\mbox{otherwise}.$$
The corresponding 1-dimensional central extension $\mathfrak{gl}_{\infty}^{*}$
and its restricted modules has been related to certain quantum vertex algebras and their
$\phi$-coordinated modules in \cite{JL14}.
}
\end{example}

\begin{example}
\label{exmp2}
{\em
Let $\mathfrak{g}=\mathcal{A}:=\mbox{span}\{E_{m,n} \ | \ m+n\in2\Z\}$.
 Then $\mathcal{A}$ is a subalgebra of $\mathfrak{gl}_{\infty}$. For $\alpha,m\in\Z$, denote by
 	$G_{\alpha,m}=E_{m+\alpha,m-\alpha}\in\mathcal{A}.$
 Then $\mathcal{A}=\mbox{span}\{G_{\alpha,m} \ | \ \alpha,m\in\Z\}$
 	with Lie bracket
 	\begin{eqnarray}
 		[G_{\alpha,m}, G_{\beta,n}]=\delta_{\alpha+\beta,m-n}G_{\alpha+\beta,\alpha+n}
 		-\delta_{\alpha+\beta,n-m}G_{\alpha+\beta,n-\alpha}
 	\end{eqnarray}
 for $\alpha,\beta,m,n\in\Z$.
 The symmetric invariant bilinear form
 on $\mathfrak{gl}_{\infty}$ in Example \ref{exmp1} restricted to $\mathcal{A}$
  is a symmetric invariant bilinear form on $\mathcal{A}$ with
  $$\langle G_{\alpha,m}, G_{\beta,n}\rangle=\delta_{\alpha+\beta,0}\delta_{m-n,0},$$
  for $\alpha,\beta,m,n\in\Z$.
 Let
 	$$\psi:\mathcal{A}\times\mathcal{A}\longrightarrow\C$$ be a bilinear map defined by
 	$$\psi(G_{\alpha,m}, G_{\beta,n})=\alpha\delta_{\alpha+\beta,0}\delta_{m-n,0}.$$
 	It is straightforward to check that $\psi$ is a 2-cocycle.
 The Lie algebra $\mathcal{A}$, vertex algebras related to $\widehat{\mathcal{A}}$ and their (equivariant) quasi modules
 are in connection with restricted modules of certain Lie algebras
 in \cite{GLTW1}, \cite{GLTW2}, \cite{LTW20}, etc.
}
\end{example}

The importance and significance of vertex algebras $V_{\widehat{\mathcal{A}^{*}}}(\ell_{2},0)$
and $V_{\widetilde{\mathcal{A}}}(\ell_{12},0)$ as well as their
certain module categories show up in the study of
the algebra of $q$-difference operators as we will see in the following two sections.

\section{Affine Lie algebra $\widehat{\mathcal{A}^{*}}$ and the algebra of $q$-difference operators}
\def\theequation{3.\arabic{equation}}
\label{sec:3}
\setcounter{equation}{0}
 In this section,
we first review the algebra of $q$-difference operators $\widetilde{V_{q}}$
(cf. \cite{BG20},
\cite{KPS94}, etc.).
 Then we show that $\widetilde{V_{q}}$ is isomorphic to the $\Z$-covariant algebra of $\widehat{\mathcal{A}^{*}}$.
  At last, we show that restricted $\widetilde{V}_{q}$-modules of level $\ell_{12}$ with $\ell_{1}=0$ are one-to-one correspond to equivariant quasi modules for the affine vertex algebra $V_{\widehat{\mathcal{A}^{*}}}(\ell_{2},0)$.

Let $q\in\C^{\times}$ be generic, i.e. $q$ is not a root of unity.
 The algebra of $q$-difference operators is the universal central extension
of the Lie algebra of inner derivations of the
quantum torus $\C_{q}[x^{\pm 1}, y^{\pm 1}]$ (\cite{KPS94}).
We rewrite its definition (and add $E_{0,0}=0$) in the following way.
\begin{definition}
 	{\em The {\em algebra of $q$-difference operators} $\widetilde{V}_{q}$ is a Lie algebra spanned by $E_{k,l}$, $c_{1}$, $c_{2}$, where $(k,l)\in\mathbb{Z}^{2}$, $c_{1}$ and $c_{2}$
 		are central elements,
 		with Lie bracket
 		\begin{eqnarray}
 			[E_{k,l},E_{r,s}]=(q^{rl-sk}-q^{sk-rl})E_{k+r,l+s}
 			+\delta_{k,-r}\delta_{l,-s}(kc_{1}+lc_{2}),\label{eq:dif}
 		\end{eqnarray}
 		for $(k,l), (r,s)\in\mathbb{Z}^{2}.$
 	}
 \end{definition}

We first give a characterization of $\widetilde{V}_{q}$ as
certain covariant algebra of the affine Lie algebra $\widehat{\mathcal{A}^{*}}$.

 For any $r\in\Z$, define
$$\sigma_{r}(G_{\alpha,m})=G_{\alpha,m+r},\;\;\sigma_{r}(K_{1})=K_{1}.$$
It is easy to see that $\sigma_{r}$ is an automorphism of $\mathcal{A}^{*}$,
 for any $r\in\Z$.
Let $\Gamma=\{\sigma_{r}\ | \ r\in\Z\}$. Then $\Gamma\cong\Z$,
and $\Gamma$ preserves $\langle\;,\;\rangle$.
$\Gamma$ extends canonically to an automorphism group of
$\widehat{\mathcal{A}^{*}}$ and furthermore to an automorphism
group of
the $\Z$-graded vertex algebra $V_{\widehat{\mathcal{A}^{*}}}(\ell_{2},0)$.
Consider the following linear character
$$\chi_{q}:\Z\longrightarrow\C^{\times};\;\;r\mapsto q^{r}
\;\;\mbox{for}\;r\in\Z.$$

We have (cf. \cite{G-KK98}, \cite{L07}):
\begin{prop}
	\label{iso2}
	The algebra of $q$-difference operators $\widetilde{V}_{q}$ is isomorphic to the $(\Z,\chi_{q})$-covariant algebra $\widehat{\mathcal{A}^{*}}[\Z]$ of the affine Lie algebra $\widehat{\mathcal{A}^{*}}$ with
	$c_{1}=\overline{K_{1}\otimes 1}$, $c_{2}=K_{2}$, and
	$$E_{\alpha,m}=\overline{G_{\alpha,0}\otimes t^{m}}\;\;\mbox{for}\;\alpha,m\in\Z.$$
\end{prop}
\begin{proof}
	From the definition, $\widetilde{V}_{q}$ has a basis
$\{E_{\alpha,m}, c_{1},c_{2}\ | \ \alpha,m\in\Z\}$ with the given bracket relations. On the other hand, from Proposition 4.4 of \cite{L07}, $\widehat{\mathcal{A}^{*}}[\Z]$ as a vector space is the quotient space of $\widehat{\mathcal{A}^{*}}$, modulo the subspace linearly spanned by
	$$\sigma_{r}(a)\otimes t^{m}-q^{-mr}(a\otimes t^{m})
\;\;\mbox{for}\;a\in\mathcal{A}^{*}, r,m\in\Z,$$
and its bracket relation is given by
	\begin{eqnarray}
		[\overline{a\otimes t^{m}},\overline{b\otimes t^{n}}]_{\Gamma}
		=\sum_{r\in\Z}q^{mr}\Big(\overline{[\sigma_{r}(a),b]\otimes t^{m+n}}+m\langle\sigma_{r}(a),b\rangle\delta_{m+n,0}K_{2}\Big),
	\end{eqnarray}
	where $\overline{a\otimes t^{m}}$ denotes the image of $a\otimes t^{m}$ in $\widehat{\mathcal{A}^{*}}[\Z]$
	for $a\in \mathcal{A}^{*}$, $m\in\Z$, and $K_{2}$ is identified with its image.
	Note that we have ($q$ is generic)
	\begin{eqnarray}
		\overline{G_{\alpha+\beta,-\alpha}\otimes t^{m+n}}
		=q^{(m+n)\alpha}\overline{G_{\alpha+\beta,0}\otimes t^{m+n}},
	\end{eqnarray}
	\begin{eqnarray}
		\overline{G_{\alpha+\beta,\alpha}\otimes t^{m+n}}
		=q^{-(m+n)\alpha}\overline{G_{\alpha+\beta,0}\otimes t^{m+n}},
	\end{eqnarray}
	\begin{eqnarray}
		\overline{K_{1}\otimes t^{m+n}}=\delta_{m+n,0}\overline{K_{1}\otimes 1}.
	\end{eqnarray}
	We see that $K_{2}$, $\overline{K_{1}\otimes 1}$ and $\overline{G_{\alpha,0}\otimes t^{m}}$ ($\alpha,m\in\Z$) form a bases of $\widehat{\mathcal{A}^{*}}[\Z]$.
	Let $\alpha,\beta,m,n\in\Z$, then
	\begin{eqnarray}
		&&{}[\overline{G_{\alpha,0}\otimes t^{m}},\overline{G_{\beta,0}\otimes t^{n}}]_{\Gamma}
		\nonumber\\
		&&{}
		=q^{m(\alpha+\beta)}\overline{G_{\alpha+\beta,\alpha}\otimes t^{m+n}}-
		q^{-m(\alpha+\beta)}\overline{G_{\alpha+\beta,-\alpha}\otimes t^{m+n}}
		\nonumber\\
		&&{}
		\;\;\;\;+\alpha\delta_{\alpha+\beta,0}\overline{K_{1}\otimes t^{m+n}}+m\delta_{\alpha+\beta,0}\delta_{m+n,0}K_{2}
		\nonumber\\
		&&{}
		=(q^{m\beta-n\alpha}-	q^{n\alpha-m\beta})\overline{G_{\alpha+\beta,0}\otimes t^{m+n}}
+\alpha\delta_{\alpha+\beta,0}\delta_{m+n,0}\overline{K_{1}\otimes 1}
		\nonumber\\
		&&{}
		\;\;\;\;+m\delta_{\alpha+\beta,0}\delta_{m+n,0}K_{2}.
	\end{eqnarray}
	It follows that $\widetilde{V}_{q}$ is isomorphic to $\widehat{\mathcal{A}^{*}}[\Z]$ with
	$E_{\alpha,m}$ corresponding to $\overline{G_{\alpha,0}\otimes t^{m}}$ for $\alpha,m\in\Z$,
	$c_{1}$ corresponding to $\overline{K_{1}\otimes 1}$ and $c_{2}$ corresponding to $K_{2}$.
\end{proof}

For $k\in\mathbb{Z}$, form a generating function
$$E_{k}(x)=\sum_{l\in\mathbb{Z}}E_{k,l}x^{-l-1}.$$
Then the defining relation (\ref{eq:dif}) can be written as
\begin{eqnarray}
&&{}[E_{k}(x_{1}),E_{r}(x_{2})]\nonumber\\
&&{}=
q^{k}E_{k+r}(q^{k}x_{2})x_{1}^{-1}\delta\Big(\frac{q^{k+r}x_{2}}{x_{1}}\Big)
- q^{-k}E_{k+r}(q^{-k}x_{2})x_{1}^{-1}\delta\Big(\frac{q^{-k-r}x_{2}}{x_{1}}\Big)
\nonumber \\
&&{}\;\;\;\;\;
+k\delta_{k,-r}x_{1}^{-1}x_{2}^{-1}\delta\Big(\frac{x_{2}}{x_{1}}\Big)c_{1}
+
\delta_{k,-r}\frac{\partial}{\partial x_{2}}x_{1}^{
	-1}\delta\Big(\frac{x_{2}}{x_{1}}\Big)c_{2}
\end{eqnarray}
for $k,r\in\mathbb{Z}$.

\begin{definition}
{\em
	A $\widetilde{V}_{q}$-module $W$ is said to be {\em restricted} if for any $w\in W$, $k\in\mathbb{Z}$, $E_{k,l}w=0$ for $l$ sufficiently large, or equivalently,
$E_{k}(x)\in\mathcal{E}(W)$ for any $k\in\Z$.
We say a $\widetilde{V}_{q}$-module $W$ is of {\em level} $\ell_{12}$
if the central elements $c_{1},c_{2}$ act as scalars $\ell_{1},\ell_{2}\in\mathbb{C}$.
}
\end{definition}

\begin{rmk}
\label{gam}
{\em Let $V$ be a $\Z$-graded vertex algebra.
Denote by $L(0)$ the degree operator on $V$.
Let $\Gamma$ be an automorphism group of the $\Z$-graded vertex algebra $V$ and
let $\chi: \Gamma \rightarrow \mathbb{C}^{\times}$ be a group homomorphism.
Then  $V$ becomes a $\Gamma$-vertex algebra
with $R_{g}=\chi(g)^{-L(0)}g$ for $g\in \Gamma$
(see \cite{L06}, \cite{L07} for the details).}
\end{rmk}

By Remark \ref{gam}, $V_{\widehat{\mathcal{A}^{*}}}(\ell_{2},0)$
becomes a $\Gamma$-vertex algebra with $\Gamma=\Z$ and
$R_{r}=q^{-rL(0)}\sigma_{r}$
for $r\in\Z$.

Now we give a connection between
affine vertex algebra $V_{\widehat{\mathcal{A}^{*}}}(\ell_{2},0)$
and restricted $\widetilde{V}_{q}$-modules.
\begin{thm}
\label{c1=0}
	Assume that $q$ is not a root of unity and let $ \ell_{2}\in\C$.
Then for any restricted $\widetilde{V}_{q}$-module $W$ of level $\ell_{12}$ with
$\ell_{1}=0$,
there exists an equivariant quasi $V_{\widehat{\mathcal{A}^{*}}}(\ell_{2},0)$-module structure $Y_{W}(\cdot,x)$ on $W$, which is uniquely determined by
	\begin{eqnarray}
	Y_{W}(K_{1},x)=0,\;\;	Y_{W}(G_{\alpha,m},x)=q^{m}E_{\alpha}(q^{m}x)\;\;\mbox{for}\;\alpha,m\in\Z.
	\end{eqnarray}
	On the other hand, for any equivariant quasi $V_{\widehat{\mathcal{A}^{*}}}(\ell_{2},0)$-module $(W,Y_{W})$ such that $Y_{W}(K_{1},x)=0$,
$W$ becomes a restricted $\widetilde{V}_{q}$-module of level $\ell_{12}$ with $\ell_{1}=0$ and
	\begin{eqnarray}
		E_{\alpha}(x)=Y_{W}(G_{\alpha,0},x)\;\;\mbox{for}\;\alpha\in\Z.
	\end{eqnarray}
\end{thm}
\begin{proof}
	Let $W$ be a restricted $\widetilde{V}_{q}$-module of level $\ell_{12}$ with $c_{1}$ acts as 0.
In view of Proposition \ref{iso2},
$W$ is a restricted $\widehat{\mathcal{A}^{*}}[\Z]$-module
with $\overline{K_{1}\otimes 1}$ acts as 0, $K_{2}$ acts as $\ell_{2}$,
 and with
	$$E_{\alpha,m}=\overline{G_{\alpha,0}\otimes t^{m}}\;\;\mbox{for}\;\alpha,m\in\Z.$$
	Note that $q$ is not a root of unity, $\chi_{q}:\Z\longrightarrow\C^{\times}$ is one-to-one. By Theorem 4.9 of \cite{L07}, there exists an quasi $V_{\widehat{\mathcal{A}^{*}}}(\ell_{2},0)$-module structure $Y_{W}(\cdot,x)$ on $W$, which is uniquely determined by
$$Y_{W}(K_{1},x)=\overline{K_{1}(x)}
(=\sum_{n\in\Z}\overline{K_{1}\otimes t^{n}}x^{-n-1}
=\overline{K_{1}\otimes 1}x^{-1})=0,$$
	$$Y_{W}(G_{\alpha,m},x)=\overline{G_{\alpha,m}(x)}
(=\sum_{n\in\Z}\overline{G_{\alpha,m}\otimes t^{n}}x^{-n-1})\;\;\mbox{for}\;\alpha,m\in\Z.$$
	For $\alpha,m\in\Z$, we have
	$$Y_{W}(G_{\alpha,m},x)=\overline{G_{\alpha,m}(x)}
	=\overline{\sigma_{m}(G_{\alpha,0})(x)}
	=q^{m}\overline{G_{\alpha,0}(q^{m}x)}
	=q^{m}E_{\alpha}(q^{m}x).$$
Hence
$$Y_{W}(\sigma_{r}(G_{\alpha,m}),x)=Y_{W}(q^{r}G_{\alpha,m},q^{r}x)\;\;\mbox{for any}\;r\in\Z.$$
Consequently, there eixts an equivariant quasi $V_{\widehat{\mathcal{A}^{*}}}(\ell_{2},0)$-module structure $Y_{W}(\cdot,x)$ on $W$
such that
$$Y_{W}(K_{1},x)=0,\;\;	Y_{W}(G_{\alpha,m},x)=q^{m}E_{\alpha}(q^{m}x)\;\;\mbox{for}\;\alpha,m\in\Z.$$
	The other direction follows from Proposition \ref{iso2}, and Theorem 4.9 of \cite{L07}.
\end{proof}

\begin{rmk}
{\em Theorem \ref{c1=0} holds for all $\ell_{1}\in\C$ if we require
a less natural condition that $Y_{W}(K_{1},x)=\ell_{1}x^{-1}$ (so that
$\overline{K_{1}\otimes 1}$ acts as $\ell_{1}$).
}
\end{rmk}

\section{Vertex algebra $V_{\widetilde{\mathcal{A}}}(\ell_{12},0)$ and $\widetilde{V_{q}}$}
\def\theequation{4.\arabic{equation}}
\label{sec:4}
\setcounter{equation}{0}

In order to naturally relate restricted $\widetilde{V_{q}}$-modules of level $\ell_{12}$ for any complex numbers $\ell_{1}$, $\ell_{2}$ to vertex algebras and their corresponding modules, in this section, we consider $\phi$-coordinated quasi modules for vertex algebras.
 	More precisely, we show that restricted $\widetilde{V}_{q}$-module of level $\ell_{12}$ are in one-to-one correspondence to $\mathbb{Z}$-equivariant $\phi$-coordinated quasi modules for the vertex algebra $V_{\widetilde{\mathcal{A}}}(\ell_{12},0)$.

 For the definitions and related results about ($\Z$-equivariant) $\phi$-coordinated quasi modules, we refer to \cite{L11} and \cite{L13}. Set $\phi=\phi(x,z)=xe^{z}\in\C((x))[[z]]$, which is fixed
 throughout this section.

 For $k\in\mathbb{Z}$, we modify the generating function $E_{k}(x)$
   by
 	$$\widehat{E}_{k}(x)=xE_{k}(x)=\sum_{l\in\mathbb{Z}}E_{k,l}x^{-l}.$$
 	Then, for any $k,r\in\mathbb{Z}$, we have
 \begin{eqnarray}
 		[\widehat{E}_{k}(x_{1}),\widehat{E}_{r}(x_{2})]
 		&=&
 \widehat{E}_{k+r}(q^{k}x_{2})\delta\Big(\frac{q^{k+r}x_{2}}{x_{1}}\Big)
 -\widehat{E}_{k+r}(q^{-k}x_{2})\delta\Big(\frac{q^{-k-r}x_{2}}{x_{1}}\Big)
   \nonumber \\
 		&&{}\;+k\delta_{k,-r}\delta\Big(\frac{x_{2}}{x_{1}}\Big)c_{1}
 +
 		\delta_{k,-r}x_{2}\frac{\partial}{\partial x_{2}}
 \delta\Big(\frac{x_{2}}{x_{1}}\Big)c_{2}.   \label{eq:com1}
 	\end{eqnarray}

Furthermore, let
$$\widetilde{E}_{k,m}(x)=\widehat{E}_{k}(q^{m}x)\;\;\mbox{for}\;k,m\in\mathbb{Z}.$$
Then
\begin{eqnarray}
	[\widetilde{E}_{k,m}(x_{1}),\widetilde{E}_{r,n}(x_{2})]
	&=&\widetilde{E}_{k+r,n+k}(x_{2})\delta\Big(\frac{q^{-m+n+k+r}}{x_{1}}\Big)
-
\widetilde{E}_{k+r,n-k}(x_{2})\delta\Big(\frac{q^{-m+n-k-r}}{x_{1}}\Big)
  \nonumber \\
	&&\;+k\delta_{k,-r}\delta\Big(\frac{q^{n-m}x_{2}}{x_{1}}\Big)c_{1}
+
	\delta_{k,-r}x_{2}\frac{\partial}{\partial x_{2}}\delta\Big(\frac{q^{n-m}x_{2}}{x_{1}}\Big)c_{2}   \label{eq:com2}
\end{eqnarray}
for any $k,m,r,n\in\mathbb{Z}$.

For $k,m\in\Z$, recall the formal operator
$$G_{k,m}(x)=\sum_{i\in\mathbb{Z}}(G_{k,m}\otimes t^{i})x^{-i-1}.$$

Then in $\widetilde{\mathcal{A}}$, for any $k,m,r,n\in\Z$, we have
\begin{eqnarray}
	&&{}[G_{k,m}(x_{1}),G_{r,n}(x_{2})]
	\nonumber\\
	&&{}
	=\delta_{-m+n+k+r,0}G_{k+r,n+k}(x_{2})x_{1}^{-1}\delta\Big(\frac{x_{2}}{x_{1}}\Big)
	-\delta_{-m+n-k-r,0}G_{k+r,n-k}(x_{2})x_{1}^{-1}
	\delta\Big(\frac{x_{2}}{x_{1}}\Big)
	\nonumber\\
	&&{}
	\;\;\;+k\delta_{k,-r}\delta_{n-m,0}x_{1}^{-1}\delta\Big(\frac{x_{2}}{x_{1}}\Big)K_{1}
	+\delta_{k,-r}\delta_{n-m,0}\frac{\partial}{\partial x_{2}}x_{1}^{-1}\delta\Big(\frac{x_{2}}{x_{1}}\Big)K_{2}. \label{eq:newcom}
\end{eqnarray}

 For any $r\in\Z$,
 define a linear map $\tau_{r}$ on $\widetilde{\mathcal{A}}$
 by
 $\tau_{r}(G_{k,m}\otimes t^{i})=G_{k,m+r}\otimes t^{i}$,
 $\tau_{r}(K_{j})= K_{j},$
for $k,m,i\in\Z$, $j=1,2$.
Then $\Gamma=\{\tau_{r}\ | \ r\in\Z\}\cong\Z$ is an automorphism
group of the Lie algebra $\widetilde{\mathcal{A}}$.
 It then gives rise to an automorphism group of the vertex algebra $V_{\widetilde{\mathcal{A}}}(\ell_{12},0)$.

On the one hand, we have:
\begin{thm}
	Let $W$ be a restricted $\widetilde{V}_{q}$-module of level $\ell_{12}$. Then there exists a $\mathbb{Z}$-equivariant $\phi$-coordinated quasi $V_{\widetilde{\mathcal{A}}}(\ell_{12},0)$-module structure $Y_{W}(\cdot, x)$
	on $W$, which is uniquely determined by
	$$Y_{W}(G_{k,m},x)=\widetilde{E}_{k,m}(x)\;\;
	\mbox{for}\;k,m\in\mathbb{Z}.$$
\end{thm}
\begin{proof}
	Since $T=\{\ G_{k,m}, {\bf 1}\ | \ k,m\in\Z\}$ generates $V_{\widetilde{\mathcal{A}}}(\ell_{12},0)$ as a vertex algebra, the uniqueness is clear. We now establish the existence.
	Set
	$$U_{W}=\{{\bf1}_{W}, \widetilde{E}_{k,m}(x)\ | \ (k,m)\in\mathbb{Z}^{2}\}\subset \mathcal{E}(W).$$
	Let $(k,m),(r,n)\in\mathbb{Z}^{2}$.
From (\ref{eq:com2}) we have
\begin{eqnarray}
(x_{1}-q^{-m+n+k+r}x_{2})(x_{1}-q^{-m+n-k-r}x_{2})(x_{1}-q^{n-m}x_{2})^{2}
[\widetilde{E}_{k,m}(x_{1}),\widetilde{E}_{r,n}(x_{2})]=0.
\end{eqnarray}
So $U_{W}$ is a quasi local subset of $\mathcal{E}(W)$.
 In view of Theorem 5.4 of \cite{L11} or Theorem 4.10 of \cite{L13},
 $U_{W}$ generates a vertex algebra $\left\langle U_{W}\right\rangle_{e}$
 under the vertex operation
	$Y_{\mathcal{E}}^{e}$
 with $W$ a $\phi$-coordinated quasi module, where
	$Y_{W}(a(x),z)=a(z)\;\;\mbox{for}\;a(x)\in\left\langle U_{W}\right\rangle_{e}.$
	
	By Lemma 4.13 of \cite{L13}, from (\ref{eq:com2}), we have
	$$\widetilde{E}_{k,m}(x)_{n}^{e}\widetilde{E}_{r,n}(x)=0\;\;\mbox{for}\;n\geq 2,$$	$$\widetilde{E}_{k,m}(x)_{1}^{e}\widetilde{E}_{r,n}(x)
=\delta_{k,-r}\delta_{n-m,0}\ell_{2}{\bf 1}_{W},$$
	\begin{eqnarray*}
		&&{}\widetilde{E}_{k,m}(x)_{0}^{e}\widetilde{E}_{r,n}(x)=
\delta_{-m+n+k+r,0}\widetilde{E}_{k+r,n+k}(x)
		\\
		&&{} \;\;-\delta_{-m+n-k-r,0}\widetilde{E}_{k+r,n-k}(x)+k\delta_{k,-r}\delta_{n-m,0}\ell_{1}{\bf 1}_{W}.
	\end{eqnarray*}
	Then by Borcherds' commutator formula we have
	\begin{eqnarray}
		&&{}[Y_{\mathcal{E}}^{e}(\widetilde{E}_{k,m}(x),x_{1}),Y_{\mathcal{E}}^{e}(\widetilde{E}_{r,n}(x),x_{2})]
		\nonumber\\
		&&{}
		=\sum_{n\geq 0}Y_{\mathcal{E}}^{e}(\widetilde{E}_{k,m}(x)_{n}^{e}\widetilde{E}_{r,n}(x),x_{2})
		\frac{1}{n!}\Big(\frac{\partial}{\partial x_{2}}\Big)^{n}x_{1}^{-1}\delta\Big(\frac{x_{2}}{x_{1}}\Big)
		\nonumber\\
		&&{}
		=\left(
\delta_{-m+n+k+r,0}Y_{\mathcal{E}}^{e}(\widetilde{E}_{k+r,n+k}(x),x_{2})
-\delta_{-m+n-k-r,0}Y_{\mathcal{E}}^{e}(\widetilde{E}_{k+r,n-k}(x),x_{2})\right)
x_{1}^{-1}\delta\Big(\frac{x_{2}}{x_{1}}\Big)
		\nonumber\\
		&&{}
		\;\;\;+k\delta_{k,-r}\delta_{n-m,0}x_{1}^{-1}\delta\Big(\frac{x_{2}}{x_{1}}\Big)\ell_{1}{\bf 1}_{W}
	+\delta_{k,-r}\delta_{n-m,0}\frac{\partial}{\partial x_{2}}x_{1}^{-1}\delta\Big(\frac{x_{2}}{x_{1}}\Big)\ell_{2}{\bf 1}_{W}.
	\end{eqnarray}
	With (\ref{eq:newcom}), we see that
	$\left\langle U_{W}\right\rangle_{e}$ is a
$\widetilde{\mathcal{A}}$-module
	of level $\ell_{12}$ with $G_{k,m}(z)$ acting as
	$Y_{\mathcal{E}}^{e}(\widetilde{E}_{k,m}(x),z)$ for $(k,m)\in\mathbb{Z}^{2}$.
	
	Let
	$\rho:V_{\widetilde{\mathcal{A}}}(\ell_{12},0)\longrightarrow \left\langle U_{W}\right\rangle_{e}$
	be a $\widetilde{\mathcal{A}}$-module homomorphism with $\rho({\bf1})={\bf 1}_{W}$.
	Then for any $(k,m)\in\mathbb{Z}^{2}$,
	$n\in\mathbb{Z}$, $v\in V_{\widetilde{\mathcal{A}}}(\ell_{12},0)$, we have
	$$\rho((G_{k,m})_{n}v)=\widetilde{E}_{k,m}(x)^{e}_{n}\rho(v)=
	\rho(G_{k,m})_{n}\rho(v).$$
It follows that $\rho$ is a homomorphism of vertex algebras.
	Consequently, $W$ becomes a $\phi$-coordinated quasi $V_{\widetilde{\mathcal{A}}}(\ell_{12},0)$-module with
	$Y_{W}(G_{k,m},x)=\widetilde{E}_{k,m}(x)$,
	$(k,m)\in\mathbb{Z}^{2}.$
	
	Moreover, for any $k,m,d\in\Z$, we have
	$$
		Y_{W}(\sigma_{r}(G_{k,m}),x)=Y_{W}(G_{k,m+r},x)=\widetilde{E}_{k,m+r}(x)
=\widetilde{E}_{k,m}(q^{r}x)
=Y_{W}(G_{k,m},\chi_{q}(\sigma_{r})x),
$$
	and it is clear that $\{Y_{W}(v,x)\ | \ v\in T\}$ is $\chi_{q}(\mathbb{Z})$-quasi local.
	Then it follows from Lemma 4.21 of \cite{L13} that
	$(W, Y_{W})$ is a $\mathbb{Z}$-equivariant $\phi$-coordinated quasi $V_{\widetilde{\mathcal{A}}}(\ell_{12},0)$-module.
\end{proof}

On the other hand, we have the following theorem.
\begin{thm}
	Let $W$ be
	a $\mathbb{Z}$-equivariant $\phi$-coordinated quasi $V_{\widetilde{\mathcal{A}}}(\ell_{12},0)$-module.
	Then $W$ is a restricted $\widetilde{V}_{q}$-module of level $\ell_{12}$ with
	$$\widetilde{E}_{k,m}(x)=Y_{W}(G_{k,m},x)\;\;
	\mbox{for}\;k,m\in\mathbb{Z}.$$
\end{thm}
\begin{proof}
	For $k,m,r\in\mathbb{Z}$, we have
	\begin{eqnarray*}
		&&{}Y_{W}(G_{k,m+r},x)=Y_{W}(\sigma_{r}(G_{k,m}),x)=
		Y_{W}(G_{k,m},\chi_{q}(\sigma_{r})x)=
		Y_{W}(G_{k,m}, q^{r}x).
	\end{eqnarray*}
	Let $(k,m),(r,n)\in\mathbb{Z}^{2}$.
For any $r\in\mathbb{Z}$, $s\geq 0$, there is
	\begin{eqnarray}
		&&{}(\sigma_{r}(G_{k,m}))_{s}G_{l,n}=(G_{k,m+r})_{s}G_{l,n}
		\nonumber\\
		&&{}
		=(G_{k,m+r})_{s}(G_{l,n})_{-1}{\bf 1}
		=[G_{k,m+r}\otimes t^{s}, G_{l,n}\otimes t^{-1}]{\bf 1}
		\nonumber\\
		&&{}
		=\delta_{s,0}[G_{k,m+r},G_{l,n}]_{-1}{\bf 1}
+\delta_{s,0}\psi(G_{k,m+r},G_{l,n})\ell_{1}{\bf 1} \nonumber\\
&&{}
\;\;\;\;+s\left\langle G_{k,m+r},G_{l,n}\right\rangle\delta_{s-1,0}\ell_{2}{\bf 1}.
	\end{eqnarray}
	Noticing that $\chi_{q}$ is injective ($q$ is not a root of unity),
	by Theorem 4.19 of \cite{L13}, we have
	\begin{eqnarray}
		&&{}[Y_{W}(G_{k,m},x_{1}), Y_{W}(G_{l,n},x_{2})]
		\nonumber\\
		&&{}
		=\mbox{Res}_{x_{0}}\sum_{r\in\mathbb{Z}}Y_{W}(Y(\sigma_{r}(G_{k,m}),x_{0})G_{l,n},x_{2})
e^{x_{0}(x_{2}\frac{\partial}{\partial x_{2}})}
		\delta\Big(\frac{\chi_{q}(\sigma_{r})x_{2}}{x_{1}}\Big)
		\nonumber\\
		&&{}
		=\sum_{r\in\mathbb{Z}}\sum_{s\geq 0}
		Y_{W}(\sigma_{r}(G_{k,m})_{s}G_{l,n},x_{2})
		\frac{1}{s!}\Big(x_{2}\frac{\partial}{\partial x_{2}}\Big)^{s}\delta\Big(\frac{q^{r}x_{2}}{x_{1}}\Big)
		\nonumber\\
		&&{}
		=\sum_{r\in\mathbb{Z}}Y_{W}([G_{k,m+r},G_{l,n}],x_{2})
\delta\Big(\frac{q^{r}x_{2}}{x_{1}}\Big)
+\sum_{r\in\mathbb{Z}}\psi(G_{k,m+r},G_{l,n})
\delta\Big(\frac{q^{r}x_{2}}{x_{1}}\Big)\ell_{1}{\bf 1}_{W}
\nonumber\\
		&&{}
		\;\;\;\;+\sum_{r\in\mathbb{Z}}\left\langle G_{k,m+r},G_{l,n}\right\rangle
		x_{2}\frac{\partial}{\partial x_{2}}
\delta\Big(\frac{q^{r}x_{2}}{x_{1}}\Big)\ell_{2}{\bf 1}_{W}
		\nonumber\\
		&&{}
		=
		Y_{W}(G_{k+l,n+k},x_{2})
		\delta\Big(\frac{q^{-m+n+k+l}x_{2}}{x_{1}}\Big)
-Y_{W}(G_{k+l,n-k},x_{2})
		\delta\Big(\frac{q^{-m+n-k-l}x_{2}}{x_{1}}\Big)
		\nonumber\\
		&&{}
		\;\;\;+k\delta_{k,-l}\delta\Big(\frac{q^{n-m}x_{2}}{x_{1}}\Big)\ell_{1}{\bf 1}_{W}
		+\delta_{k,-l}x_{2}\frac{\partial}{\partial x_{2}}\delta\Big(\frac{q^{n-m}x_{2}}{x_{1}}\Big)\ell_{2}{\bf 1}_{W}.
	\end{eqnarray}
	With (\ref{eq:com2}), we see that $W$ is a $\widetilde{V}_{q}$-module of level $\ell_{12}$ with
	$\widetilde{E}_{k,m}(x)=Y_{W}(G_{k,m},x)$, for $(k,m)\in\mathbb{Z}^{2}.$
	And $Y_{W}(G_{k,m},x)\in\mathcal{E}(W)$ for $(k,m)\in\mathbb{Z}^{2}.$
	Therefore, $W$ is a restricted $\widetilde{V}_{q}$-module of level $\ell_{12}$.
	\end{proof}

We then wish to study the generalized affine vertex algebra
$V_{\widetilde{\mathcal{A}}}(\ell_{12},0)$ and its simple quotient
$L_{\widetilde{\mathcal{A}}}(\ell_{12},0)$.
For this, there is the following result.
\begin{prop}
\label{dirsum}
The 2-cocycle $\psi$ of $\mathcal{A}$ in Example \ref{exmp2}
gives trivial central extension of $\mathcal{A}$.
Similarly, the 2-cocycle $\psi_{2}$ in Remark \ref{p2} gives trivial central extension
of the affine Lie algebra $\widehat{\mathcal{A}}$.
\end{prop}
\begin{proof}
Let $\mu:\mathcal{A}\longrightarrow\C$ be a linear map defined by
$$\mu(G_{\alpha,m})=\frac{1}{2}\delta_{\alpha,0}m,$$
for $\alpha,m\in\Z$.
Then
$\psi(G_{\alpha,m},G_{\beta,n})=\mu([G_{\alpha,m},G_{\beta,n}])$
for any $\alpha,\beta,m,n\in\Z$.

 Let
$\mu_{2}:\widehat{\mathcal{A}}\longrightarrow\C$ be a linear map defined by
$$\mu_{2}(G_{\alpha,m}\otimes t^{i})=\frac{1}{2}\delta_{\alpha,0}\delta_{i+1,0}m,\;\;
\mbox{and}\; \mu_{2}(K_{2})=0,$$
 for $\alpha,m,i\in\Z$.
Then
$\psi_{2}(G_{\alpha,m}\otimes t^{i},G_{\beta,n}\otimes t^{j})
=\mu_{2}([G_{\alpha,m}\otimes t^{i},G_{\beta,n}\otimes t^{j}])$,
for any $\alpha,\beta,m,n\in\Z$.

\end{proof}

Therefore, by Proposition \ref{iso}, we have
$L_{\widetilde{\mathcal{A}}}(\ell_{12},0)\cong L_{\widehat{\mathcal{A}}}(\ell_{2},0)$
for any $\ell_{1},\ell_{2}\in\C$.
 Vertex algebra $L_{\widehat{\mathcal{A}}}(\ell_{2},0)$
and its modules has been studied in Section 4 of \cite{LTW20}.

\section{Category $\mathcal{O}$ of $\widetilde{V}_{q}$-modules}
	\def\theequation{5.\arabic{equation}}
\label{sec:5}
	\setcounter{equation}{0}

In this section, we introduce and study a category $\mathcal{O}$
of $\widetilde{V}_{q}$-modules.
We prove that objects in the category $\mathcal{O}$
are restricted $\widetilde{V}_{q}$-modules.
We also classify simple modules in the category $\mathcal{O}$.

Let
$$b=
\bigoplus_{l\geq 0,k\in\mathbb{Z}}\mathbb{C}E_{k,l}
\oplus\mathbb{C} c_{1}\oplus\mathbb{C} c_{2}.$$
Let $\ell_{1},\ell_{2}\in\mathbb{C}$.
Given a $b$-module
$V$ with $c_{1}$, $c_{2}$ act
as scalars $\ell_{1},\ell_{2}$ respectively.
Consider the induced module
$$\mbox{Ind}_{\ell_{12}}(V)=U(\widetilde{V_{q}})\otimes_{U(b)}V.$$

We first show that under certain conditions the induced module
$\mbox{Ind}_{\ell_{12}}(V)$ is a simple $\widetilde{V_{q}}$-module.
\begin{thm}
	\label{thm:simple}
	Let $V$ be a simple $b$-module and assume that there exists $t\in\mathbb{Z}_{+}$ such that
	\begin{itemize}
		\item [(1)] $E_{k,t}$ acts injectively on $V$ for all $k\in\Z$.
		\item [(2)] $E_{k,l}V=0$ for all $k\in\Z$, $l> t$.
	\end{itemize}
Then for any $\ell_{1},\ell_{2}\in\mathbb{C}$, the induced module
{\em $\mbox{Ind}_{\ell_{12}}(V)$} is a simple $\widetilde{V}_{q}$-module.
\end{thm}
\begin{proof}
Let $v$ be a nonzero element in $V$.
Let $k^{'},k_{1}\in\Z$, $j_{1}\in\Z_{+}$.
It is straightforward to show (by induction) that for any
 $i_{1}\in\N$  we have
$$E_{k^{'},t+j}E^{i_{1}}_{k_{1},-j_{1}}v\in
\C E_{k^{'}+i_{1}k_{1},t+j-i_{1}j_{1}}v\subset V
\;\;\mbox{for all}\;j\geq i_{1}j_{1}.$$
Let $n\geq 1$, $k_{1}, k_{2},\ldots, k_{n}\in\Z$, $j_{1},j_{2},\ldots,j_{n}\in\Z_{+}$.
Similarly, one can get that for any
$i_{1},i_{2},\ldots,i_{n}\in\N$ we have
$$E_{k^{'},t+j}E_{k_{n},-j_{n}}^{i_{n}}\cdots
E_{k_{2},-j_{2}}^{i_{2}}E_{k_{1},-j_{1}}^{i_{1}}v
\in \C E_{k^{'}+\sum_{l=1}^{n}i_{l}k_{l},t+j-(\sum_{l=1}^{n}i_{l}j_{l})}v\subset V$$
for all $j\geq \sum_{l=1}^{n}i_{l}j_{l}.$
In particular,
$$E_{k^{'},t+j}E_{k_{n},-j_{n}}^{i_{n}}\cdots
E_{k_{2},-j_{2}}^{i_{2}}E_{k_{1},-j_{1}}^{i_{1}}v=0\;\;\mbox{for}\;
j>\sum_{l=1}^{n}i_{l}j_{l}$$
and
$$E_{k^{'},t+\sum_{l=1}^{n}i_{l}j_{l}}E_{k_{1},-j_{n}}^{i_{n}}\cdots
E_{k_{2},-j_{2}}^{i_{2}}E_{k_{1},-j_{1}}^{i_{1}}v
\in \C E_{k^{'}+\sum_{l=1}^{n}i_{l}k_{l},t}v\subset V.$$

By PBW theorem, every nonzero element $v$ of $\mbox{Ind}_{\ell_{12}}(V)$ can be uniquely written in the form of a finite sum
$w=\sum\limits_{m\geq 0}E_{k_{m},j_{m}}\cdots E_{k_{2},j_{2}} E_{k_{1},j_{1}}v_{m},$
 where
$v_{m}\in V$.
With above and the assumptions (1) and (2), for any $w\in\mbox{Ind}_{\ell_{12}}(V)$,
we can always arrive at a nonzero element in $V$.
The simplicity of $V$ then tells us that $\mbox{Ind}_{\ell_{12}}(V)$ is a simple $\widetilde{V}_{q}$-module.
\end{proof}

Recall that a module $V$ over a Lie algebra $L$
is {\em locally finite}
		if for any $v\in V$, $\mbox{dim}(\sum\limits_{n\in\Z_{+}} L^{n}v)<+\infty$.

For any $t\in\Z_{+}$,
denote by
$$\widetilde{V_{q}}^{(t)}=\sum_{k\in\Z,l\geq t}\C E_{k,l}.$$
It is a subalgebra of $\widetilde{V_{q}}$.
Let $\mathcal{O}$ be a category of $\widetilde{V_{q}}$-modules
which are locally finite with respect to some $\widetilde{V_{q}}^{(t)}$ for $t\in\Z_{+}$.

The following Lemma tells us that all
the modules in the category $\mathcal{O}$
are restricted modules.
\begin{lem}
\label{zero}
Let $V$ be a $\widetilde{V_{q}}$-module in the category $\mathcal{O}$.
Then
for any nonzero vector $v\in V$, there exists $t\in\Z_{+}$ such that
$\widetilde{V_{q}}^{(t)}v=0$.
\end{lem}
\begin{proof}
Let $0\neq v\in V$.
By assumption, there exists $r\in\Z_{+}$ such that
$\widetilde{V_{q}}^{(r)}$ acts locally finite on $V$.
Hence $W=U(\widetilde{V_{q}}^{(r)})v$ is a finite dimensional
$\widetilde{V_{q}}^{(r)}$-module.
Let $I\subset \widetilde{V_{q}}^{(r)}$ be the kernel
of the representation map, i.e.
$I=\{x\in\widetilde{V_{q}}^{(r)}\ | \ x.v=0\}$.
Then $I$ is an ideal of $\widetilde{V_{q}}^{(r)}$ of finite codimension.
We claim that
there exists $k,l\in\Z$,
$l\geq r$, such that $E_{k,l}\in I$.
If not, then there exists a minimal $m\in\N$ such that
$I$ contains
$$E_{\underline{n},\underline{m}}:=(a^{(s)}_{1}E_{k_{11},s}
+\cdots+a^{(s)}_{n_{1}}E_{k_{1n_{1}},s})
+\cdots+(a^{(s+m)}_{1}E_{k_{m1},s+m}
+\cdots+a^{(s+m)}_{n_{m}}E_{k_{mn_{m}},s+m})$$
for some $s\geq r$, $k_{ij}\in\Z$, $i=1,\ldots,m$, $j=1,\ldots,n_{1},\ldots,n_{m}$,
and complex numbers $a^{(l)}_{j}$ satisfying $a^{(s)}_{1}\neq 0$,
$a^{(s)}_{n_{1}}\neq 0$ if $m=0$ (and then $n_{1}$ is taken to be the minimal one) or $a^{(s+m)}_{n_{m}}\neq 0$ if $m\in\Z_{+}$.
Then
$I$ contains
$[E_{k_{11},s},E_{\underline{n},\underline{m}}]$ which is an element
of the form
$(b^{(s)}_{2}E_{k_{11}+k_{12},2s}
+\cdots+b^{(s)}_{n_{1}}E_{k_{11}+k_{1n_{1}},2s})
+\cdots+(b^{(s+m)}_{1}E_{k_{11}+k_{m1},2s+m}
+\cdots+b^{(s+m)}_{n_{m}}E_{k_{11}+k_{mn_{m}},2s+m}).$
If $m=0$, then it contradicts to our choice of $n_{1}$.
If $m\in\Z_{+}$, then
$I$ contains
$$\displaystyle[E_{\sum_{i=1}^{n_{1}-1}2^{n_{1}-1-i}k_{1i}+k_{1n_{1}},2^{n_{1}-1}s},\cdots,[E_{k_{11}+k_{12},2s},[E_{k_{11},s},
E_{\underline{n},\underline{m}}]]\cdots],$$
which is of the form
$(c^{(s+1)}_{1}E_{r_{21},(2^{n_{1}}-1)s+s+1}
+\cdots+c^{(s+1)}_{n_{2}}E_{r_{2n_{2}},(2^{n_{1}}-1)s+s+1})
+\cdots+(c^{(s+m)}_{1}E_{r_{m1},(2^{n_{1}}-1)s+s+m}
+\cdots+c^{(s+m)}_{n_{m}}E_{r_{mn_{m}},(2^{n_{1}}-1)s+s+m}),$
this contradicts to our choice of $m$.

Let $E_{k,l}\in I$.
Then $l\geq r$, $k\in\Z$, $E_{k,l}v=0$.
We claim then $\widetilde{V_{q}}^{(2r+l)}\subset I$.
For any $E_{m,n}\in\widetilde{V_{q}}^{(2r+l)}$.
Then $n\geq 2r+l$.
We have
$$[E_{m-k,n-l-r}, E_{k,l}]=(q^{k(n-r)-lm}-q^{lm-k(n-r)})E_{m,n-r}.$$
{\bf Case 1:} If $k(n-r)= lm$.
Take a nonzero integer $s$ with $ns\neq mr$ (for latter use).
Then
$$[E_{m-k-s,n-l-r}, E_{k,l}]=(q^{sl}-q^{-sl})E_{m-s,n-r}$$
gives that $E_{m-s,n-r}\in I$.
And then
$$[E_{m-s,n-r}, E_{s,r}]=(q^{ns-mr}-q^{mr-ns})E_{m,n}$$
gives that $E_{m,n}\in I$.\\
{\bf Case 2:} If $k(n-r)\neq lm$, then
$E_{m,n-r}\in I$.
And then if $m\neq 0$,
$$[E_{m,n-r}, E_{0,r}]=(q^{-mr}-q^{mr})E_{m,n}$$
gives that $E_{m,n}\in I$.
For $m=0$. If $k\neq 0$, then
$$[E_{-k,n-l}, E_{k,l}]=(q^{kn}-q^{-kn})E_{0,n}$$
gives that $E_{0,n}\in I$.
For $m=0$ and $k=0$, we are in Case 1.
Hence the result holds.
\end{proof}


As we need, we introduce the following notion.
If a $\widetilde{V_{q}}$-module $V$ is generated by a vector $0\neq v\in V$ with
$E_{k,l}v=0$ for all $k\in\Z$ and $l>0$, then $V$ is called a
 {\em highest weight module}. Harish-Chandra modules in \cite{LT06}
 are highest weight modules.

Finally, we give a classification of all simple $\widetilde{V_{q}}$-module
in the category $\mathcal{O}$.
\begin{thm}
Let $S$ be a simple $\widetilde{V_{q}}$-module
in the category $\mathcal{O}$.
Then $S$ is a highest weight module, or there exists $\ell_{1},\ell_{2}\in\C$, $t\in\Z_{+}$ and a simple
$b$-module $V$ such that both conditions (1) and (2) of Theorem \ref{thm:simple}
are satisfied and {\em $S\cong\mbox{Ind}_{\ell_{12}}V$}.
\end{thm}
\begin{proof}
By Lemma \ref{zero}, for any nonzero $v\in S$, there exists
some $j\in\N$ such that $v$ is annihilated by all $E_{k,l}$,
$k\in\Z$, $l> j$.
Consider the following vector space
$$N_{j}=\{v\in S\ | \ E_{k,l}v=0\;\;\mbox{for all}\;k\in\Z,l> j\}.$$
We have $N_{j}\neq 0$ for some $j$ by the above argument.
If $N_{0}\neq 0$, then $S$ is an irreducible highest weight module.
Assume $j\geq 1$.
Thus we can find a smallest positive integer, say $t$, with $V:=N_{t}\neq 0$.
It is easy to check that $V$ is a $b$-module.
Note that $c_{1},c_{2}$ act as scalars on $V$, say $\ell_{1},\ell_{2}$ respectively.
It follows from $t$ is smallest that assumption (1) of Theorem \ref{thm:simple}
is satisfied.

Then there is a canonical surjective (since $S$ is simple) morphism
$$\pi:\mbox{Ind}_{\ell_{12}}V\longrightarrow S,\;\;1\otimes v\mapsto v,\;\;\forall v\in V.$$
It remains to show that $\pi$ is also injective.
Let $K=\mbox{Ker}(\pi)$.
If $K\neq 0$. Take $0\neq v\in K$. Then $v\in \mbox{Ind}_{\ell_{12}}V$.
 By the same analysis as the proof of Theorem \ref{thm:simple}, from $v$ we can arrive at $0\neq u\in V$.
Since $K$ is a $\widetilde{V_{q}}$-module, $u\in K$.
Hence $u\in V\cap K$.
But $V\cap K=0$, contradiction!
Hence $K=0$.
$S\cong\mbox{Ind}_{\ell_{12}}V$ as $\widetilde{V_{q}}$-modules.
Then $V$ is a simple $b$-module by the property of induced modules.
\end{proof}

\end{document}